\documentclass[reqno]{amsart}

\usepackage{amssymb, amsmath, amsthm}
\usepackage{color}

\newcommand*{\fplus}{\genfrac{}{}{0pt}{}{}{+}}
\newcommand*{\fdots}{\genfrac{}{}{0pt}{}{}{\cdots}}
\newcommand*{\fminus}{\genfrac{}{}{0pt}{}{}{-}}

\newcommand{\qrfac}[2]{{\left({#1}; q\right)_{#2}}} 

\renewcommand{\Im}{\operatorname{Im}}

\newcommand{\qbin}{\genfrac{[}{]}{0pt}{}}

\newcommand{\smallo}{o}
\newcommand{\elliptictheta}[1]{\theta\!\left({#1} ;q \right) }

\newtheorem{Theorem}{Theorem}[section]
\newtheorem{Proposition}[Theorem]{Proposition}

\newtheorem*{rems}{Remarks} 
\newenvironment{Remarks}{\begin{rems}\normalfont}{\end{rems}}

\newtheorem*{rem}{Remark} 
\newenvironment{Remark}{\begin{rem}\normalfont}{\end{rem}}

\numberwithin{equation}{section}

\usepackage{hyperref}
\hypersetup{
    colorlinks=true, 
    linktoc=all,     
    linkcolor=blue,  
    citecolor=blue
}

\begin{document} 
\title[Orthogonal polynomials from continued fractions]{Orthogonal polynomials associated with a continued fraction of Hirschhorn}

\author[G.~Bhatnagar]{Gaurav Bhatnagar
}
\address{Department of Mathematics, Ashoka University, Sonipat, Haryana 131029, India
}
\email{bhatnagarg@gmail.com}

\author[M.~E.~H.~Ismail]{Mourad~E.~H.~Ismail}
\address{Department of Mathematics, University of Central Florida, Orlando, FL 32816, USA  }
\email{ismail@math.ucf.edu}

\dedicatory{Dedicated to the memory of Dick Askey}

\date{\today}

\begin{abstract}
We study orthogonal polynomials associated with a continued fraction due to Hirschhorn. 
Hirschhorn's continued fraction contains as special cases 
the famous Rogers--Ramanujan continued fraction and two of Ramanujan's generalizations. 
The orthogonality measure of the set of 
polynomials obtained has an absolutely continuous component. We find generating functions, asymptotic formulas, orthogonality relations, and the Stieltjes transform of the measure. Using standard generating function techniques, we show how to obtain formulas for the convergents of Ramanujan's continued fractions, including a formula that Ramanujan recorded himself as Entry 16 in Chapter 16 of his second notebook.
\end{abstract}

\keywords{Continued fractions, Orthogonal polynomials, Rogers-Ramanujan Continued Fraction,  the Lost Notebook}
\subjclass[2010]{Primary 33D45; Secondary 30B70}

\maketitle

\section{Introduction}



The connection of continued fractions with  orthogonal polynomials is well known. Indeed,  orthogonal polynomials made an appearance in the context of continued fractions as early as 1894, in the work of Stieltjes \cite{TJS1894, TJS1895}.  
Our objective in this paper is to 
study orthogonal polynomials associated to
a continued fraction due to Hirschhorn \cite{MDH1974} (also considered by Bhargava and Adiga \cite{BA1984}). 
This continued fraction is
\begin{equation}\label{mikecf} 
\frac{1}{1-b+a}\fplus\frac{b+\lambda q}{1-b+aq}\fplus\frac{b+\lambda q^2}{1-b+aq^2}\fplus\frac{b+\lambda q^3}{1-b+aq^3}\fplus\fdots,
\end{equation}
where we have changed a few symbols in order to fit the notation used by 
Andrews and Berndt~\cite{AB2005} in their edited version of Ramanujan's Lost Notebook. 
Hirschhorn's continued fraction contains three of Ramanujan's famous continued fractions. 
For example, when one of $a$ or $b$ is $0$, it reduces to continued fractions in the Lost Notebook (see
 \cite{AB2005} and \cite[Ch.\ 16]{Berndt1991-RN3}); when both $a$ and $b$ are $0$ it is the Rogers--Ramanujan continued fraction. 

%

Some of the orthogonal polynomials arising from these special cases have been studied before. A set of orthogonal polynomials corresponding to the $b=0$ case have been studied previously by Al-Salam and Ismail \cite{Al-I1983}, and we study one more. The orthogonality measure of the polynomials associated to Hirschhorn's continued fraction studied here has an absolutely continuous component, as opposed to the discrete measure in \cite{Al-I1983}. 

The techniques we use were developed by Askey and Ismail in their memoir \cite{AI1984}.
 These authors study classical orthogonal polynomials using techniques involving recurrence relations, generating functions and asymptotic methods. They also used a theorem of Nevai \cite{PN1979}.  In addition, we apply a moment method 
 developed by Ismail and Stanton~\cite{IS2002}. In the context of Ramanujan's continued fractions, these ideas have been applied previously by Al--Salam and Ismail \cite{Al-I1983} and Ismail and Stanton \cite{IS2006}. 

The contents of this paper are as follows. In Section \ref{sec:mikecf1} we provide some  
background information from the theory of orthogonal polynomials. 
In Section \ref{sec:Mike-Nevai}, we translate 
Hirschhorn's continued fraction to a form suitable for our study. Further, we apply Nevai's theorem to compute a formula for the absolutely continuous component of the measure of the orthogonal polynomials associated with Hirschhorn's continued fraction. 
In Section \ref{sec:stieltjes}, we obtain another expression for this by inverting the associated Stieltjes transform. 
In Section~\ref{sec:moments}, we provide another solution of the recurrence relation consisting of 
functions that are moments over a discrete measure.  
In Section~\ref{sec:specialcases}, we consider a continued fraction of Ramanujan
obtained by taking $b=0$ in Hirschhorn's continued fraction, where we obtain a discrete 
orthogonality measure. 

Finally, in Section~\ref{sec:convergents}, we show how to obtain formulas for the convergents of a continued fraction. Such a formula was given by Ramanujan himself, who gave a formula for the convergents of 
\begin{equation}
\frac{1}{1}\fplus\frac{\lambda q}{1}\fplus\frac{\lambda q^2}{1}\fplus\frac{\lambda q^3}{1}\fplus\fdots,
\end{equation}
the Rogers--Ramanujan continued fraction. His formula for its convergents appears as Entry 16 in Chapter 16 of Ramanujan's second notebook (see Berndt 
\cite{Berndt1991-RN3}). 

To state Ramanujan's formula, we require some notation. We need the $q$-rising factorial $\qrfac{q}{n}$, which is defined to be $1$ when $n=0$; and 
$$\qrfac{q}{n} = (1-q)(1-q^2)\cdots (1-q^{n}),$$
for $n$ a positive integer.

Next we have the $q$-binomial coefficient, defined as
$$\qbin{n}{k}_q = \frac{\qrfac{q}{n}}{\qrfac{q}{k}\qrfac{q}{n-k}}$$
where $n\geq k$ are nonnegative integers. When $n<k$ we take $\qbin{n}{k}_q=0.$

Ramanujan's formula  is as follows.  
\begin{equation}
\frac{N_n}{D_n} =
\frac{1}{1}\fplus \frac{\lambda q}{1}\fplus
\frac{\lambda q^2}{1}\fplus\frac{\lambda q^3}{1}
\fplus\fdots \fplus \frac{\lambda q^n}{1}
 ,\label{entry16}
\end{equation}
%
%
where
\begin{equation*}\label{num-entry16}
N_n = 
\sum_{k\geq 0} q^{k^2+k}\lambda^k
\qbin{n-k}{k}_{q}
\end{equation*}
and
\begin{equation*}\label{den-entry16}
D_n = 
\sum_{k\geq 0} q^{k^2}\lambda^k 
\qbin{n-k+1}{k}_{q}
.
\end{equation*}
The sums $N_n$ and $D_n$ are finite sums. For example, the summand of $N_n$ is $0$ when the index $k$ is such that $n-k<k$.

Unlike the work of Ramanujan, there will be no mystery about how such formulas are discovered. 

\section{Background: From continued fractions  to orthogonal polynomials}
\label{sec:mikecf1}
If there is a continued fraction, then there is a three-term recurrence relation. And if the recurrence relation is of the \lq right type\rq, it defines a set of orthogonal polynomials.  Such a recurrence relation is central to the study of the associated orthogonal polynomials and examining it directs our study. The objective of this section is to collate this background information from the theory of orthogonal polynomials. We have used Chihara \cite{Chihara1978} and the second author's book \cite{MI2009}. For introductory material on these topics we recommend Andrews, Askey and Roy \cite{AAR1999}.

The right type of continued fraction is called the $J$-fraction, which is of the form
\begin{equation}\label{jfrac}
\frac{A_0}{A_0x+B_0}\fminus\frac{C_1}{A_1x+B_1}\fminus\frac{C_2}{A_2x+B_2}\fminus\fdots.
\end{equation}
The $k$th convergent of the $J$-fraction is given by
$$\frac{N_k(x)}{D_k(x)} :=
\frac{A_0}{A_0x+B_0}\fminus\frac{C_1}{A_1x+B_1}\fminus\fdots \fminus \frac{C_{k-1}}{A_{k-1}x+B_{k-1}}.
$$
The following proposition shows how to compute the convergents of a continued fraction. 
\begin{Proposition}[{\cite[Th. 2.6.1, p.\ 35]{MI2009}}]\label{cf-conv} Assume that $A_kC_{k+1}\neq 0$, $k=0, 1, \dots.$ Then the polynomials $N_k(x)$ and $D_k(x)$ are solutions of the recurrence relation
\begin{equation} \label{three-term}
y_{k+1}(x) = (A_kx+B_k) y_k(x) -C_ky_{k-1}(x), \text{ for } k>0,
\end{equation}
with the initial values 
$$D_0(x)=1, D_1(x)=A_0x+B_0, N_0(x)=0, N_1(x)=A_0.$$ 
\end{Proposition}
Instead of \eqref{three-term}, we consider a three term recurrence equation of the form
\begin{equation} \label{three-term2}
x y_{k}(x) = y_{k+1}(x) + \alpha_k y_k(x) +\beta_k y_{k-1}(x), \text{ for } k>0,
\end{equation}
where $\alpha_k$ is real for $k\geq 0$ and $\beta_k>0$ for $k>0$. 
This three-term recurrence can be obtained from \eqref{three-term} by mildly re-scaling the 
functions involved. 

Let the polynomials $\left\{ P_k(x)\right\}$ satisfy \eqref{three-term2}, with the initial values  
$$P_0(x)=1  \text{ and }P_1(x)=x-\alpha_0.$$  We will also have occasion to consider the polynomials $\left\{ P_k^*(x)\right\}$, satisfying \eqref{three-term2} with the initial conditions
$P_0^*(x)=0$ and $P_1^*(x)=1$. The $P_k^*(x)$ correspond to the numerator and $P_k(x)$ to the denominator of the associated $J$-fraction. Note that both $P_k^*(x)$ and  $P_k(x)$ are monic polynomials, of degree $k-1$ and $k$, respectively. 

The next proposition shows that the $P_k(x)$ 
are orthogonal with respect to a measure $\mu$. 
 \begin{Proposition}[Spectral Theorem {\cite[Th.\ 2.5.2]{MI2009}}]\label{spectral}
Given a sequence $\{P_n(x)\}$ as above, there is a positive measure $\mu$ such that
 $$\int_a^b P_n(x)P_m(x) d\mu(x) =  \beta_1\beta_2\dots\beta_n \cdot \delta_{mn}.$$
 \end{Proposition}
 \subsection*{Some pertinent facts}
\begin{itemize}
\item The interval [a, b] is the convex hull of the support of $\mu$. 
\item  If $\{\alpha_k\}$ and $\{\beta_k\}$ are bounded, then the support of $\mu$ is bounded, and $[a,b]$ is a finite interval. In addition, the measure $\mu$ is unique. 
\item 
The interval [a, b] is called the interval of orthogonality and as the degree $n \to \infty$, the smallest zeros of $p_n$ converges to $a$ while the largest converges to $b$. 
\item The measure $\mu$ could possibly have both discrete and an absolutely continuous component. The orthogonality relation is then of the form
 $$\int_a^b P_n(x)P_m(x)\mu^{\prime}(x)dx + \sum_j P_n(x_j)P_m(x_j) w(x_j) = h_n\delta_{mn},$$
 where $x_j$ are the points where $\mu$ has mass $w(x_j)$, and $h_n>0$.
\item  (Blumenthal's Theorem \cite[Th.\ IV-3.5, p.~117]{Chihara1978} (rephrased)) If $\alpha_k\to \alpha$ and $\beta_k\to 0$, then the measure of the orthogonal polynomials defined by \eqref{three-term2} is purely discrete. However, if $\alpha_k\to\alpha$ and $\beta_k\to \beta>0$, then $\mu$ has an absolutely continuous component. 
\end{itemize}

Next, we have  a proposition that shows the connection between the continued fraction and the Stieltjes transform of the measure. 
 \begin{Proposition}[Markov, see {\cite[Th.\ 2.6.2]{MI2009}}]\label{prop:stieltjes} Assume that the true interval of orthogonality $[a,b]$ is bounded. Then
 $$\displaystyle
 \lim_{k\to\infty} \frac{P_k^*(x)}{P_k(x)}   =\int_a^b \frac{d\mu(t)}{x-t},$$ 
  uniformly for $x\not\in \operatorname{supp}(\mu)$.
 \end{Proposition}
 From here, we can use Stieltjes' inversion formula (see \cite[Eq.\ (1.2.9)]{MI2009}) to obtain a formula for $d\mu$.  Let 
 $$X(x)= \int_a^b \frac{d\mu(t)}{x-t},\text{ where } x\not\in \operatorname{supp}(\mu).$$ 
   Then
  $$\mu(x_2)-\mu(x_1) = \lim_{\epsilon\to 0^+} 
  \int_{x_1}^{x_2} \frac{X(x-i\epsilon)-X(x+i\epsilon)}{2\pi i} dx.$$
 So $\mu^\prime$ exists at $x$, and we have \cite[Eq.\  (1.2.10)]{MI2009}:
 \begin{equation}\label{stieltjes-inverse}
 \mu^{\prime}(x)=  \frac{X(x-i0^+)-X(x+i0^+)}{2\pi i}. 
 \end{equation}
%

To summarize, each $J$-fraction is associated with a three-term recurrence relation. The solutions of a (possibly scaled) three-term recurrence relation, under certain conditions, are orthogonal polynomials. Both the numerator and denominator of the continued fraction satisfy the recurrence relation, with differing initial conditions. The limit of their ratio, that is the value of the continued fraction, gives a formula for the orthogonality measure of the denominator polynomials. 


\section{Hirschhorn's Continued Fraction: computing the measure}\label{sec:Mike-Nevai}

In this section, we begin our study of the orthogonal polynomials associated with  Hirschhorn's continued fraction. On examining the associated three-term recurrence relation, we find that the associated denominator polynomials have an orthogonality measure with an absolutely continuous component. 
Our goal in this section is to 
compute a formula for this, using a very useful theorem of 
Nevai~\cite{PN1979}. Nevai's theorem requires finding the asymptotic expression for the denominator polynomials, for which we will use Darboux's method.

We need some notation. The {\em $q$-rising factorial} $\qrfac{a}{n}$ is defined as
$$\qrfac{a}{n} := 
\begin{cases}
1 & \text{ for } n=0\cr
(1-a)(1-aq)\cdots (1-aq^{n-1}) & \text{ for } n=1, 2, \dots.
\end{cases}
$$ In addition
$$\qrfac{a}{\infty} := \prod_{k=0}^\infty (1-aq^{k}) \text{ for } |q|<1.$$
We use the short-hand notation
\begin{align*}
\qrfac{a_1, a_2,\dots, a_r}{k} &:= \qrfac{a_1}{k} \qrfac{a_2}{k}\cdots
\qrfac{a_r}{k}.
\end{align*}

We now begin our study of Hirschhorn's continued fraction by considering the more general $J$-fraction
\begin{equation}
H(x):= 
\frac{1-b}{x(1-b)+a}\fplus\frac{b+\lambda q}{x(1-b)+aq}\fplus 
\frac{b+\lambda q^2}{x(1-b)+aq^2}\fplus
\fdots   \label{mikecf-jfrac}
\end{equation}
Note that \eqref{mikecf} is $H(1)/(1-b)$. 
On comparing with the form of the $J$-fraction in \eqref{jfrac} we find that 
$$A_k= (1-b), B_k=aq^k \text{ for } k=0, 1, 2, \dots \text{ and }C_k = -(b+\lambda q^k) \text{ for } k=1, 2, 3, \dots.$$ The corresponding three term recurrence relation is
\begin{equation}\label{mikecf-3term}
y_{k+1}(x) = (x(1-b) + aq^k) y_k(x) + (b+\lambda q^{k}) y_{k-1}(x), \text{ for } k > 0.
\end{equation}
By Proposition~\ref{cf-conv}, the numerator and denominator polynomials (denoted by $N_n(x)$ and $D_n(x)$) satisfy \eqref{mikecf-3term} and the initial values 
$$D_0(x)=1, D_1(x)=x(1-b)+a, N_0(x)=0, N_1(x)= 1-b.$$ 
On writing \eqref{mikecf-3term} in the form \eqref{three-term2}, by replacing $x$ by $x/(1-b)$,
we note that
$\beta_k=-(b+\lambda q^k)\to -b$, so if $b<0$ the measure has an absolutely continuous 
component. 


%
%

We use a theorem of Nevai  \cite[Th.\ 40, p.\ 143]{PN1979} (see \cite[Th. 11.2.2, p.\ 294]{MI2009})
 to find the absolutely continuous component of the measure. 
\begin{Proposition}[Nevai]\label{Nevai}
Assume that the set of orthogonal polynomials $\{P_k(x)\}$ are as in Proposition~\ref{spectral}. If
\begin{equation}\label{ineq:nevai}
\sum_{k=1}^\infty \left( \left| \sqrt{\beta_k}-\frac{1}{2}\right| +|\alpha_k|\right) <\infty,
\end{equation}
then $\mu$ has an absolutely continuous component $\mu^\prime$ supported on $[-1,1]$. 
Further, if $\mu$ has a discrete part, then it will lie outside $(-1,1)$. 
In addition, the limiting relation
\begin{equation}\label{eq:nevai}
\limsup_{k\to\infty} \left( \frac{P_k(x)\sqrt{1-x^2}}{\sqrt{\beta_1\beta_2\cdots\beta_{k}}} -
\sqrt{\frac{2\sqrt{1-x^2}}{\pi\mu^{\prime}(x)}} \sin\left((k+1)\vartheta -\phi(\vartheta)\right)
\right)
=0
\end{equation}
holds, with $x=\cos\vartheta \in (-1,1)$. Here $\phi(\vartheta)$ does not depend on $k$. 
\end{Proposition}
\begin{Remark}
The interval $[-1,1]$ need not be the true interval of orthogonality. 
\end{Remark}
%
We first modify the recurrence relation \eqref{mikecf-3term} so that the hypothesis of 
Proposition~\ref{Nevai} is satisfied. 
Let
$$P_k(x):=\frac{y_k(\gamma x)}{\gamma^k(1-b)^k},$$
where $\gamma$ will be determined shortly. 
Next, divide \eqref{mikecf-3term} by $\gamma^{k+1}(1-b)^{k+1}$ to see that $P_k(x)$ satisfies the recurrence 
\begin{equation*}
xP_{k}(x)=P_{k+1}(x)-\frac{aq^k}{\gamma(1-b)} P_k(x)-\frac{b+\lambda q^k}{\gamma^2(1-b)^2}P_{k-1}(x).
\end{equation*}
Recall that $b<0$. We now choose $$\gamma^2 = -\frac{4b}{(1-b)^2}$$
to find that the recurrence reduces to
\begin{equation}\label{mike-p-3term2}
xP_{k}(x)=P_{k+1}(x)+cq^k P_k(x)+\frac1{4} \left(1+\lambda q^k/b\right) P_{k-1}(x),
\end{equation}
with $c=-a/2{\sqrt{-b}}$.  

Motivated by the above considerations, we consider the polynomials defined by \eqref{mike-p-3term2} with the initial conditions $P_0(x)=1$ and $P_1(x)=x-c$, so that $P_k(x)$ satisfy the three-term recurrence \eqref{three-term2} with $\alpha_k=cq^k$, $\beta_k=\left(1+\lambda q^k/b\right)/4$. 

A short calculation shows that the conditions for Proposition~\ref{Nevai} are satisfied.
 Assume that $0<|q|<1$. Then for $k$ large enough, we can see using the mean value theorem that
\begin{align*}
 \left| \sqrt{\beta_k}-\frac{1}{2} \right| +|\alpha_k | 
 & = 
 \left| \frac{1}{2}\left(1+\lambda q^k/b\right)^{1/2}-\frac{1}{2} \right| +|cq^k | \cr
&\leq  C \left|  \frac{\lambda q^k}{4b} \right| +|cq^k | 
\end{align*}
for some constant $C$. 
Thus 
\begin{equation*}
\sum_{k=1}^\infty  \left( \left| \sqrt{\beta_k}-\frac{1}{2} \right| +|\alpha_k |\right) <\infty.
\end{equation*}
The choice of $\gamma$ is now transparent.

The idea is to compare the asymptotic expression for $P_k(x)$ with \eqref{eq:nevai}  to determine the formula for $\mu^\prime(x)$. For this purpose we will use Darboux's method, which can be stated as follows. 
\begin{Proposition}[Darboux's Method ({see \cite[Th.\ 1.2.4]{MI2009}})]\label{darboux}
Let $f(z)$ and $g(z)$ be analytic in the disk $\{ z: |z|<r\}$ and assume that 
$$f(z)=\sum_{k=0}^{\infty} f_kz^k, \text{ } g(z)=\sum_{k=0}^{\infty} g_kz^k, \text{ } |z|<r.$$
If $f-g$ is continuous on the closed disk $\{ z: |z|\leq r\}$ then
$$f_k = g_k +o\left(r^{-k}\right).$$
\end{Proposition}

%

With these preliminaries, we now proceed with our first result, the orthogonality relation for $P_k(x)$. 
\begin{Theorem} 
Let $q$ be real satisfying $0<|q|<1$, $c\in \mathbb{R} $, and $1+\lambda q^k/b >0$.
Let $P_k(x)$ be a set of polynomials defined by \eqref{mike-p-3term2} satisfying the initial conditions $P_0(x)=1$ and $P_1(x)=x-c$. 
Then we have the orthogonality relation:
\begin{equation*}
\int P_n(x)P_m(x)d\mu =
\frac{1}{4^n}\qrfac{-\lambda q/b}{n} 
\delta_{mn},
\end{equation*}
where $\mu$ has an absolutely continuous component, and
\begin{align}\label{mu-prime}
\mu^{\prime}(x) &=
\frac{2}{\pi}
\frac{\qrfac{-\lambda q/b}{\infty}}{|R|^2\sqrt{1-x^2} } \text{ for } x\in (-1,1),
\\
\intertext{with}
R&=\frac{-1}{i\sin\vartheta}
\sum_{m=0}^{\infty} \frac{\qrfac{- \lambda qe^{i\vartheta} /2bc}{m} }{\qrfac{q, qe^{2i\vartheta} }{m}}
(-2c)^m e^{im\vartheta} q^{m\choose 2} ,
\end{align}
and $x=\cos\vartheta$. Further, if $\mu$ has a discrete part,  it will lie outside $(-1,1)$.
\end{Theorem}
\begin{Remarks}\ 
\begin{enumerate}
\item
We can take $x=\cos\vartheta$ and write the part of the integral where $\mu$ has an absolutely continuously component as follows:
\begin{equation*}
\frac{2\qrfac{-\lambda q/b}{\infty} }{\pi}
\int_{-1}^1 \frac{P_n(x)P_m(x)}{\sqrt{1-x^2} |R|^2}dx=
\frac{2\qrfac{-\lambda q/b}{\infty} }{\pi}
\int_{0}^\pi \frac{P_n(\cos\vartheta)P_m(\cos\vartheta)}{|R|^2}d\vartheta .
\end{equation*}
\item
The denominator polynomials we considered in Section~\ref{sec:mikecf1} are related to $P_k(x)$ as follows:
$$P_k(x)=\frac{D_k(\gamma x)}{\gamma^k(1-b)^k}.$$
\end{enumerate}
\end{Remarks}
\begin{proof}
We have already seen that the hypothesis for Nevai's theorem are satisfied. 

To use Darboux's method to find the formula for $P_k(x)$, we require its generating function. 
Let $P(t)$ denote the generating function of $P_k(x)$, that is,
$$P(t):= \sum_{k=0}^\infty P_k(x)t^k.$$
Multiply \eqref{mike-p-3term2} by $t^{k+1}$ and sum over $k\geq 0$ to find that
$$P(t) = \frac{1}{1-xt+t^2/4} - \frac{ct(1+\lambda tq /4bc)}{1-xt+t^2/4} P(tq).$$
We change the variable by taking $$x=\frac{e^{i\vartheta}+e^{-i\vartheta}}{2} \text{ } (= \cos \vartheta)$$
so
\begin{align*}
1-xt+t^2/4 &= (1-e^{i\vartheta}t/2)(1-e^{-i\vartheta}t/2)\cr
&= (1-\alpha t)(1-\beta t) .
\end{align*}
Using $\alpha$ and $\beta$ we can write the $q$-difference equation for $P(t)$ in the form 
\begin{align*}
P(t) &= \frac{1}{\qrfac{\alpha t, \beta t }{1} } - \frac{ct(1+\lambda tq/4bc)}
{\qrfac{\alpha t, \beta t }{1} } P(tq)
\cr
&= \sum_{k=0}^{\infty} \frac{\qrfac{-\lambda tq /4bc}{k} }{\qrfac{\alpha t, \beta t }{k+1}}
(-ct)^{k}q^{k\choose 2},
\end{align*}
by iteration.
So we obtain
\begin{equation}\label{mikecf2-gf}
P(t) = \sum_{k=0}^{\infty} \frac{\qrfac{-\lambda tq /4bc}{k} }{\qrfac{e^{i\vartheta} t/2, e^{-i\vartheta} t/2 }{k+1}}
(-ct)^{k}q^{k\choose 2}.
\end{equation}

Next we use Darboux's method to find an asymptotic expression for $P_k(x)$, where 
$x=\cos\vartheta$. The terms in the denominator are 
$$(1-e^{i\vartheta }t/2)(1-e^{i\vartheta}tq/2)\cdots (1-e^{-i\vartheta}t/2)(1-e^{i\vartheta}tq/2)\cdots.$$
The poles are at 
$$t=2e^{-i\vartheta}, 2e^{-i\vartheta}/q, 2e^{-i\vartheta}/q^2,\dots; \text{ and }
t=2e^{i\vartheta}, 2e^{i\vartheta}/q, 2e^{i\vartheta}/q^2\dots.$$ 
Since $0<|q|<1 $, the poles nearest to $t=0$ are at $t=2e^{-i\vartheta}$ and 
$t=2e^{i\vartheta}$.
We consider
\begin{equation}\label{mikecf2-Q}
Q(t):= \sum_{m=0}^{\infty} 
\frac{\qrfac{- \lambda qe^{i\vartheta} /2bc}{m} }
{\qrfac{q, qe^{2i\vartheta} }{m}}
\frac{(-2ce^{i\vartheta})^{m}q^{m\choose 2} }{1-e^{2i\vartheta}}\frac{1}{1-\frac{t}{2}e^{-i\vartheta}}
\end{equation}
and observe that $P(t)-Q(t)$ has a removable singularity at $t=2e^{i\vartheta}.$
Similarly, the we consider the conjugate
\begin{equation}\label{milecf2-Q-conj}
\overline{Q}(t)= \sum_{m=0}^{\infty} \frac{\qrfac{- \lambda qe^{-i\vartheta} /2bc}{m} }{\qrfac{q, qe^{-2i\vartheta} }{m}}
\frac{(-2ce^{-i\vartheta})^{m}q^{m\choose 2} }{1-e^{-2i\vartheta}}\frac{1}{1-\frac{t}{2}e^{i\vartheta}}
\end{equation}
and note that $P(t)-\overline{Q}(t)$ has a removable singularity at $t=2e^{-i\vartheta}.$
Thus we see that 
$$P(t)-Q(t)-\overline{Q}(t)$$
is continuous in $|t|\leq 2$. Thus Darboux's method can be used to find the formula for $P_k(x)$. Writing $$Q(t)=\sum_{k=0}^\infty Q_k t^k$$ we see that 
$$P_k(x) =Q_k+\overline{Q}_k + \smallo(2^{-k}).$$
The geometric series implies that
$$Q_k = \frac{e^{-i(k+1)\vartheta}}{2^{k+1} (-i)\sin\vartheta}
\sum_{m=0}^{\infty} \frac{\qrfac{- \lambda qe^{i\vartheta} /2bc}{m} }{\qrfac{q, qe^{2i\vartheta} }{m}}
(-2c)^m e^{im\vartheta} q^{m\choose 2} 
$$
and
$$\overline{Q}_k = \frac{e^{i(k+1)\vartheta}}{2^{k+1} i\sin\vartheta}
\sum_{m=0}^{\infty} \frac{\qrfac{- \lambda qe^{-i\vartheta} /2bc}{m} }{\qrfac{q, qe^{-2i\vartheta} }{m}}
(-2c)^m e^{-im\vartheta} q^{m\choose 2} .
$$
We denote by $R$ the part of $Q_k$ that is independent of $k$. That is, let
$$R:=\frac{-1}{i\sin\vartheta}
\sum_{m=0}^{\infty} \frac{\qrfac{- \lambda qe^{i\vartheta} /2bc}{m} }{\qrfac{q, qe^{2i\vartheta} }{m}}
(-2c)^m e^{im\vartheta} q^{m\choose 2} 
$$
and write it in a form
$$R=|R|e^{i\phi}.$$
It is clear that $\phi$ is independent of $k$ (though it depends on $\vartheta$). Now using this notation
we have the asymptotic formula for $P_k(x)$:
\begin{align}
P_k(x) &\sim Q_k+ \overline{Q}_k 
=\frac{|R|}{2^{k}}  
\sin\left((k+1)\vartheta-\phi+\frac{\pi}{2}\right). \label{formula-pk}
\end{align}


Now that we know the asymptotic formula for $P_k(x)$ we can compare with \eqref{eq:nevai} and obtain the expression for the measure $\mu^\prime$. To do so, we informally write \eqref{eq:nevai} as
$$\frac{P_k(x)\sqrt{\mu^\prime(x)}}{\sqrt{\beta_1\beta_2\cdots \beta_{k}}} \sim
\sqrt{\frac{2}{\pi}}  \frac{\sin\left( (k+1)\vartheta -\phi(\vartheta)\right)}{(1-x^2)^{1/4}}
.$$
Note that 
$$\beta_1\beta_2\cdots\beta_k=\frac{1}{4^k}\qrfac{-\lambda q/b}{k}.$$

Comparing \eqref{formula-pk} with the above, we find that
\begin{align*}
\mu^\prime (x) 
= \frac{2\qrfac{-\lambda q/b}{\infty}}{\pi \sqrt{1-x^2} |R|^2} \label{mu-prime}
\end{align*}
where $x=\cos\vartheta$.
In this manner, we have obtained an expression for $\mu^{\prime}$ from Nevai's theorem. This completes the proof. 
\end{proof}

\begin{Remarks}
We can take the special cases $\lambda=0$ and $a=0$ in \eqref{mikecf} and obtain analogous results for the 
corresponding special cases of Hirschhorn's continued fractions. The special case $a=0$ is a 
continued fraction considered by Ramanujan in his Lost Notebook, see 
\cite[Entry 6.3.1(iii)]{AB2005}. However, if we take $b=0$, 
 $\mu$ does not have an absolutely convergent component. We consider this case in 
 Section~\ref{sec:specialcases}.
 \end{Remarks}

\section{The Stieltjes Transform}\label{sec:stieltjes}

We now recall Proposition \ref{prop:stieltjes} which says that the continued fraction is given,
for $x\not\in \operatorname{supp}(\mu)$, by  the Stieltjes transform of the measure $\mu$. In this section, we provide the evaluation of the continued fraction. In addition, we invert the Stieltjes transform using \eqref{stieltjes-inverse}, and obtain an alternate expression for $\mu^\prime$. 

Recall the notation $P_k^*(x)$ for the polynomials satisfying 
\eqref{mike-p-3term2} with initial conditions $P^*_0(x)=0$ and $P^*_1(x)=1$. The polynomials 
$P_k(x)$ satisfy the same recurrence with the initial conditions
$P_0(x)=1$ and $P_1(x)=x-c$.
We need to compute, for $x\not\in \operatorname{supp}(\mu)$,
\begin{equation}\label{cf-mike-transformed}
X(x)= \lim_{k\to\infty} \frac{P_k^*(x)}{P_k(x)}.  
\end{equation}
Again we will  appeal to Darboux's theorem. However, this time the computation of the formula 
is not for $x\in (-1,1)$ but for $x\in \mathbb{C}\setminus\operatorname{supp}(\mu)$.
 


Note that since $x=\cos\vartheta$,  $$e^{\pm i\vartheta} = x\pm \sqrt{x^2-1}.$$
We choose a 
branch 
 of $\sqrt{x^2-1}$ in such a way that
$$\sqrt{x^2-1} \sim x, \text{ as } x\to\infty,$$
so that
$\left| e^{-i\vartheta}\right|<\left|e^{i\vartheta}\right|$
for $x$ in the upper half plane, and 
$\left|e^{i\vartheta}\right|<\left|e^{-i\vartheta}\right|$
for $x$  in the lower half plane. 
We use the notation  $\rho_1=e^{-i\vartheta}$ and $\rho_2=e^{i\vartheta}$. 

\begin{Theorem}\label{th:cf-values} Let $X(x)$ be the continued fraction in \eqref{cf-mike-transformed}. Let $\rho_1$ and $\rho_2$ be as above.  Let $F$ and $G$ be defined as follows:
\begin{align*}
F(\rho) &= \sum_{m=0}^{\infty} \frac{\qrfac{- \lambda q\rho /2bc}{m} }{\qrfac{q, q\rho^2 }{m}}
(-2c\rho)^{m}q^{\binom{m+1}{2}}, \\
\intertext{and}
G(\rho) &= \sum_{m=0}^{\infty} \frac{\qrfac{- \lambda q\rho /2bc}{m} }{\qrfac{q, q\rho^2 }{m}}
(-2c\rho)^{m}q^{\binom{m}{2}} .
\end{align*}
Then  $X(x)$ converges for all complex numbers $x\not\in (-1,1)$, except possibly a finite set of 
points,  and is given by
$$X(x)= 2\rho\frac{F(\rho)}{G(\rho)}, $$
where $\rho$ is given by:
$$\rho = 
\begin{cases}
 \rho_1, & \text{if } \Im(x)>0 , \text{ or } x > 1 \text{ ($x$ real)} \cr
 \rho_2,  & \text{if } \Im(x)<0, \text{ or } x < -1 \text{ ($x$ real)}  \cr
1, & \text{if } x = 1, \cr
 -1, & \text{if } x = -1. 
\end{cases}
$$
\end{Theorem}
\begin{proof}

We first compute asymptotic formulas for $P_k^*(x)$ and $P_k(x)$ in the upper half plane. It is not difficult to see that the generating function of $P_k^{*}(x)$ is given by
\begin{equation}\label{mikecf2-gf-num}
P^*(t) = \frac{t}{(1-\rho_1 t/2) (1-\rho_2 t/2)}
\sum_{k=0}^{\infty} \frac{\qrfac{-\lambda tq /4bc}{k} }{\qrfac{\rho_1 qt/2, \rho_2 qt/2 }{k}}
(-ct)^{k}q^{\binom{k+1}2}.
\end{equation}
When $x$ is in the upper half-plane, the singularity nearest the origin is at $t=2\rho_1$. 
Let $Q^*(t)$ be the series
$${Q^*}(t)= 
\frac{2\rho_1}{(1-\rho_1^2) (1-\rho_2 t/2)}
\sum_{m=0}^{\infty} \frac{\qrfac{- \lambda q\rho_1 /2bc}{m} }{\qrfac{q, q\rho_1^2 }{m}}
(-2c\rho_1)^{m}q^{\binom{m+1}{2}}. $$
Then $P^*(t)-Q^*(t)$ has a removable singularity at $t=2\rho_1$. By Darboux's method, 
we have
$$P_k^*(x) \sim 
\frac{2\rho_1 \rho_2^k}{2^k(1-\rho_1^2) }
\sum_{m=0}^{\infty} \frac{\qrfac{- \lambda q\rho_1 /2bc}{m} }{\qrfac{q, q\rho_1^2 }{m}}
(-2c\rho_1)^{m}q^{\binom{m+1}{2}}
= \frac{2\rho_1 \rho_2^k}{2^k(1-\rho_1^2) } F(\rho_1) . $$
Similarly, considering the generating function of $P_k(x)$ when $x$ is in  the upper half plane, we find that
$$P_k (x) \sim 
\frac{\rho_2^k}{2^k(1-\rho_1^2) }
\sum_{m=0}^{\infty} \frac{\qrfac{- \lambda q\rho_1 /2bc}{m} }{\qrfac{q, q\rho_1^2 }{m}}
(-2c\rho_1)^{m}q^{\binom{m}{2}}
= \frac{\rho_2^k}{2^k(1-\rho_1^2) } G(\rho_1) . $$
Thus for $x$ in the upper half-plane, we find that
$$X(x)=\lim_{k\to\infty} \frac{P_k^*(x)}{P_k(x)} =2\rho_1\frac{F(\rho_1)}{G(\rho_1)}.$$
The same calculation works when $x$ is real, and $x>1$.

When $x$ is  in the lower half-plane, since $t=2\rho_2$ is the singularity nearest to the origin, 
a similar calculation yields
$$X(x)=\lim_{k\to\infty} \frac{P_k^*(x)}{P_k(x)} =2\rho_2\frac{F(\rho_2)}{G(\rho_2)}.$$
This is also valid for real values of $x$ such that $x<-1$.

For $x=1$, we find that the generating function for $P_k^*(1)$ is given by
\begin{equation*}
P^*(t) = \frac{t}{(1- t/2)^2}
\sum_{k=0}^{\infty} \frac{\qrfac{-\lambda tq /4bc}{k} }{\qrfac{qt/2}{k}^2}
(-ct)^{k}q^{\binom{k+1}2}.
\end{equation*}
The singularity nearest the origin is at $t=2$. The dominating term of the comparison function 
is given by
$${Q^*}(t)= 
\frac{2}{(1-t/2)^2}
\sum_{m=0}^{\infty} \frac{\qrfac{- \lambda q /2bc}{m} }{\qrfac{q, q }{m}}
(-2c)^{m}q^{\binom{m+1}{2}}. $$
(The singular part of $P^*(t)$ has an additional term of the form $A/(1-t/2)$, but that does not contribute to $P_k^*(1)$.)
Darboux's method yields
$$P_k^*(1) \sim 
\frac{2(k+1)}{2^k }
\sum_{m=0}^{\infty} \frac{\qrfac{- \lambda q /2bc}{m} }{\qrfac{q, q }{m}}
(-2c)^{m}q^{\binom{m+1}{2}}
= \frac{2 (k+1)}{2^k } F(1) . $$
Similarly, we find that 
$$P_k(1)\sim \frac{(k+1)}{2^k } G(1), $$
and so
$$X(1) = 2\frac{F(1)}{G(1)},$$
as required. The computation at $x=-1$ is similar. 

In the above, we cannot have $G(\rho)=0$. Replace $F$ and $G$ by $\qrfac{q\rho^2}{\infty}F$ and $\qrfac{q\rho^2}{\infty}G$. Now $G$ is an entire function so has only finitely many zeros in any bounded set. 
The zeros of this modified $G$ are the mass points of the discrete part of the measure.   For a further explanation of why there may be a finite set of points outside of $(-1,1)$ where $X(x)$ does not converge, see the remarks at the end of the section.
\end{proof}

On inverting the Stieltjes Transform using \eqref{stieltjes-inverse}, we have another formula for the absolutely continuous component of the orthogonality measure.
\begin{Theorem} Let $\mu^{\prime}$ be given by \eqref{mu-prime} and let $F$ and $G$ be as in 
Theorem~\ref{th:cf-values}. Then, for $x\in (-1,1)$, we have
\begin{equation*}
\mu^{\prime}(x) = \frac{1}{\pi i}
\left(
\rho_2\frac{F(\rho_2)}{G(\rho_2)}-\rho_1\frac{F(\rho_1)}{G(\rho_1)}
\right)
.
\end{equation*}
\end{Theorem}

\begin{proof}
From Theorem~\ref{th:cf-values}, it follows that
in the upper half-plane,
$$X(x+i0^+)=\lim_{k\to\infty} \frac{P_k^*(x)}{P_k(x)} =2\rho_1\frac{F(\rho_1)}{G(\rho_1)},$$
where now $x$ is a real number in $(-1,1)$. 
Similarly, we have 
$$X(x-i0^+)=\lim_{k\to\infty} \frac{P_k^*(x)}{P_k(x)} =2\rho_2\frac{F(\rho_2)}{G(\rho_2)}.$$
The theorem now follows from \eqref{stieltjes-inverse}.
\end{proof}

%

\begin{Remarks} 
Before closing this section, we make a few remarks concerning 
 the discrete part of the measure $\mu$. Recall that Nevai's theorem says that the discrete part of the measure will lie outside $(-1,1)$. Let $X(x)=F/G $ represent the continued fraction, with $F$ and $G$ entire functions (as above). 
 Assume that 
$x_0$ is an isolated mass point of weight $m_0$. 
Then $X(x)$ is of the form
\begin{align*}
X(x) &=\frac{F}{G} =\int \frac{d\mu(t)}{x-t}\cr
&= \int_{-1}^1 \frac{\mu^{\prime }d t}{x-t} + 
\frac{m_0}{x-x_0} + \text{terms from other isolated mass points} 
.
\end{align*}
Thus $X(x)$ has a simple pole at $x_0$ with residue equal to $m_0$. Since the measure is positive, the residue $m_0$ is positive. 
\begin{enumerate}
\item 
This implies that the mass points of the discrete part of the measure occur at the poles of the continued fraction $X(x)$ outside $(-1,1)$. 
 The function $G$ is an entire function of $\rho$ of order zero so it must have infinitely many zeros. However it can only have finitely many zeros in the unit disc, hence $X(x)$ has finitely many poles because we chose $|\rho| < 1$. 

\item We can show that the zeros of $F(x)$ interlace with the zeros of $G(x)$. 
The poles of $X(x)$ occur at the zeros of $G$. If the pole is at $x=\rho$, 
we have 
$$m=\frac{F(\rho)}{G^{\prime}(\rho)}>0.$$
Thus, $F$ and $G^{\prime}$ have the same sign. Now at two successive zeros of
$G(x)$, the sign of $G^{\prime}(x)$ will be different. And thus the sign of $F(x)$ changes at two successive zeros of $G(x)$. This implies that $F$ has a zero between two successive  zeros of $G$.  
\end{enumerate}
\end{Remarks}
Unfortunately, we are unable to compute the zeros of $G$ from our formulas, and thus cannot say 
much more about the discrete part of the measure. 
%

\section{Solutions of the recurrence that are moments}\label{sec:moments}

Recall the definition of the $q$-integral:
\begin{equation*}
\int_{a}^b f(t) d_q t:= b(1-q) \sum_{n=0}^\infty q^n f(bq^n) -
a(1-q) \sum_{n=0}^\infty q^n f(aq^n).
\end{equation*}
In this section we find a solution $p_k(x)$ of \eqref{mike-p-3term2} of the form
\begin{equation}\label{pk-def}
p_k(x) = \int_{t_1}^{t_2} t^k f(t) d_q t,
\end{equation}
following a technique developed by Ismail and Stanton in \cite{IS1997, IS1998, IS2002}. 

We will use the integration by parts formula
\begin{equation}\label{qbyparts}
\int_a^b f(t) g(qt) d_q t = \frac{1}{q} \int_a^b g(t) f(t/q) d_q t
+\frac{1-q}{q} \big( ag(a)f(a/q) - bg(b)f(b/q)\big).
\end{equation}
This formula follows from the definition of the $q$-integral. 

We will require the notation of {\em basic hypergeometric series} (or $_r\phi_s$ series). This series is of the form
\begin{equation*}
_{r}\phi_s \left[\begin{matrix} 
a_1,a_2,\dots,a_r \\
b_1,b_2,\dots,b_s\end{matrix} ; q, z
\right] :=
\sum_{k=0}^{\infty} \frac{\qrfac{a_1,a_2,\dots, a_r}{k}}{\qrfac{q, b_1,b_2,\dots, b_s}{k}}
\left( (-1)^kq^{\binom k2}\right)^{1+s-r} z^k.
\end{equation*}
When $r=s+1$, the series converges for $|z|<1$. See Gasper and Rahman~\cite{GR90} for further convergence conditions for these series. 

\begin{Theorem} 
Let $|\lambda q/b|<1$.  With $x=\cos\vartheta$, we define $p_k(x)$ as the $q$-integral
\begin{multline}
p_k(x)  \label{final-pk-qint}
:=
 \frac
 {4 (-i\sin\vartheta)}
 {(1-q)}
\frac
{\qrfac{ 2c e^{i\vartheta}, 2c e^{-i\vartheta} }{\infty}}
{\qrfac{q, e^{2i\vartheta},  e^{-2i\vartheta} }{\infty}}\cr
\times
\int_{\frac{1}{2}e^{-i\vartheta}}^{\frac{1}{2}e^{i\vartheta}} t^k 
\frac{\qrfac{2q e^{i\vartheta} t,2q e^{-i\vartheta} t,  -\lambda q/4bct}{\infty}}
{\qrfac{ 4c t, q/4c t}{\infty}}
d_q t.
\end{multline}
Then 
$p_k(x)$ satisfies the recurrence relation  \eqref{mike-p-3term2}. 

Further, let $|\lambda q/2bc|<1$. Then, for $\Im (x) \ge 0$, we have
\begin{subequations}
\begin{align}
\label{final-pka}
p_k(x) &=
\frac{e^{ik\vartheta}\qrfac{2ce^{-i\vartheta}}{k} \qrfac{-\frac{\lambda q}{2bc }e^{-i\vartheta}}{\infty}}{2^{k}
\qrfac{ \frac{q}{2c}e^{-i\vartheta}}{\infty}}
 \ \!
 _{2}\phi_1
\left[ \begin{matrix} 
 -{b}q^{-k}/{\lambda},  0
\\
q^{1-k} e^{i\vartheta}/2c 
\end{matrix}
;q, -\frac{\lambda q}{2bc}e^{-i\vartheta}
\right]\\
\intertext{ and, for $\Im (x) \le 0 $, we have}
p_k(x) &=
\frac{e^{-ik\vartheta}\qrfac{2ce^{i\vartheta}}{k} \qrfac{-\frac{\lambda q}{2bc }e^{i\vartheta}}{\infty}}{2^{k}
\qrfac{ \frac{q}{2c}e^{i\vartheta}}{\infty}}
 \ \!
 _{2}\phi_1
\left[ \begin{matrix} 
 -{b}q^{-k}/{\lambda},  0
\\
q^{1-k} e^{-i\vartheta}/2c
\end{matrix}
;q, -\frac{\lambda q}{2bc}e^{i\vartheta}
\right]. \label{final-pkb}
\end{align}
\end{subequations}
\end{Theorem}
\begin{Remarks}
\ 
\begin{enumerate}
\item The ratio
$p_k(x)/p_0(x),$
is a solution of \eqref{mike-p-3term2} with value $1$ at $k=0$. 
\item
When 
$b=-\lambda$, the $_2\phi_1$ in \eqref{final-pka} (and \eqref{final-pkb}) terminates,
 and  we find that $p_0(x)=1$ and $p_1(x)$ is a polynomial of 
degree $1$. Indeed, we see that $p_1(x)=x-c$, so the initial conditions will match those satisfied by 
the denominator polynomials $P_k(x)$  (with $b=-\lambda$)  considered in 
Section \ref{sec:Mike-Nevai}. In that case,  our calculations are a special case of the calculations in 
Ismail and Stanton \cite{IS1997} in their proof of Theorem 2.1(B). 
\end{enumerate}
\end{Remarks}

\begin{proof}


For now, we call our solution $g_k(x)$ and assume it satisfies \eqref{pk-def}. We will show how one can guess $f(t)$, and the limits $t_1$ and $t_2$. 
From the recurrence relation 
\eqref{mike-p-3term2}, we must have
\begin{align}\label{qint1}
x\int_{t_1}^{t_2}  t^k f(t) d_q t & =
	 \int_{t_1}^{t_2} t^{k+1} f(t) d_q t 
	+ c\int_{t_1}^{t_2} (qt)^{k} f(t) d_q t \cr
& \hspace{1in}
	+ \frac{1}{4}\int_{t_1}^{t_2} t^{k-1} f(t) d_q t
	+\frac{\lambda q}{4b}\int_{t_1}^{t_2} (qt)^{k-1} f(t) d_q t\cr
&= \int_{t_1}^{t_2} t^{k} \left(t f(t) + f(t)/4t\right) d_q t  +\cr
&\hspace{1in}	 
	\int_{t_1}^{t_2} t^{k} \left(c f(t/q)/q  +\lambda f(t/q)/4bt\right) d_q t, 
\end{align}
where we use \eqref{qbyparts} and assume that
$$f(t_1/q) =0 = f(t_2/q)$$
in the last step. Now \eqref{qint1} will be satisfied if
\begin{equation}\label{f-feqn}
f(t)\big(x-t-{1}/{4t}\big) = f(t/q) \big(c/q+\lambda/4bt\big),
\end{equation}
or
\begin{align*}
f(t) &=
 \frac{-b (1-\alpha t)(1-\beta t)}{\lambda(1+4bct/\lambda )} f(tq),
\end{align*}
where $\alpha$ and $\beta$ are such that
$$1-4qxt +4q^2t^2 =(1-\alpha t)(1-\beta t).$$
For convenience we change the variable by taking
$$x=\cos\vartheta = \frac{e^{i\vartheta} + e^{-i\vartheta}}{2}$$
so
$$\alpha = 2qe^{i\vartheta} \text{ and } \beta = 2qe^{-i\vartheta}.$$
Now if we find a function $h(t)$ such that
\begin{equation}\label{h-feqn}
h(t)=\frac{-b}{\lambda}h(tq),
\end{equation}
then we can write $f$ as
$$f(t)=\frac{\qrfac{\alpha t,\beta t}{\infty}}{\qrfac{-4bct/\lambda}{\infty}} h(t).$$

To find an $h(t)$ which satisfies \eqref{h-feqn}, we turn to the elliptic theta factorials, defined for 
$z\neq 0$ and $|q|<1$ as follows:
$$\elliptictheta{z}:=\qrfac{z, q/z}{\infty}.$$
Note the {\em quasiperiodicity} property
$$\frac{\elliptictheta{z}}{\elliptictheta{zq}}=-z.$$
This suggests that we can take $h(t)$ of the form
\begin{align*}
h(t)=\frac{\elliptictheta{At}}{\elliptictheta{Bt}}\cr
\intertext{so that} 
\frac{h(t)}{h(tq)} = \frac{A}{B}.
\end{align*}
We postpone the selection of $A$ and $B$ until later, but assume that
$$\frac{A}{B}=  \frac{-b}{\lambda},$$
so that \eqref{h-feqn} is satisfied. 

Thus, with $A$ and $B$ as above, we find a solution $f(t)$ of \eqref{f-feqn} given by
%
\begin{equation}\label{f-final}
f(t) = 
\frac{\qrfac{2q e^{i\vartheta} t,2q e^{-i\vartheta} t, At, q/At}{\infty}}{\qrfac{-4bct/\lambda, B t, q/B t}{\infty}}
.
\end{equation} 

It remains to find $t_1$ and $t_2$.  If we take 
$$t_1= \frac{1}{2}e^{-i\vartheta}, t_2= \frac{1}{2}e^{i\vartheta}$$
we will find that 
$$f(t_1/q)=0=f(t_2/q).$$ 

In this manner, we obtain an expression for a solution of the recurrence relation \eqref{mike-p-3term2} in the form \eqref{pk-def}:
\begin{equation}\label{gk-qint}
g_k(x)= \int_{\frac{1}{2}e^{-i\vartheta}}^{\frac{1}{2}e^{i\vartheta}} t^k 
\frac{\qrfac{2q e^{i\vartheta} t,2q e^{-i\vartheta} t, At, q/At}{\infty}}{\qrfac{-4bct/\lambda, B t, q/B t}{\infty}}
d_q t.
\end{equation}
We will specify $A$ and $B$ shortly. 

Using the definition of the $q$-integral, and some elementary algebraic manipulations,
we  obtain another expression for $g_k(x)$: 
\begin{align*}
g_k(x) &=
(1-q)
\frac{e^{i(k+1)\vartheta}}{2^{k+1}} 
\frac
{\qrfac{q, q e^{2i\vartheta}, \frac{A}{2}e^{i\vartheta}, \frac{2q}{A}e^{-i\vartheta}}{\infty}}
{\qrfac{ -\frac{2bc}{\lambda}e^{i\vartheta}, \frac{B}{2}e^{i\vartheta}, \frac{2q}{B}e^{-i\vartheta}}{\infty}}\cr
&\hspace{20pt}
\times 
\sum_{n=0}^\infty
\frac
{\qrfac{-\frac{2bc}{\lambda}e^{i\vartheta} }{n}}
{\qrfac{q, q e^{2i\vartheta}}n}
\left(\frac{-\lambda q^{k+1}}{b}\right) ^n\cr
& \hspace{50 pt} -
\text{ (same term with $\vartheta\mapsto -\vartheta$)}.
\end{align*}
Now using the $_r\phi_s$ notation, and 
collecting common terms, 
we can write this as
\begin{align}\label{pk-sum}
g_k(x) &=
(1-q)
\frac{e^{i(k+1)\vartheta}}{2^{k+1}} 
\frac
{\qrfac{q, q e^{2i\vartheta}, \frac{A}{2}e^{i\vartheta}, \frac{2q}{A}e^{-i\vartheta}}{\infty}}
{\qrfac{ -\frac{2bc}{\lambda}e^{i\vartheta}, \frac{B}{2}e^{i\vartheta}, \frac{2q}{B}e^{-i\vartheta}}{\infty}}\cr
&\hspace{20pt}
\times 
\Bigg( \ \!
 _{2}\phi_1
\left[ \begin{matrix} 
 -\frac{2bc}{\lambda}e^{i\vartheta},  0
\\
q e^{2i\vartheta} 
\end{matrix}
;q, -\frac{\lambda q^{k+1}}{b}
\right]
 \cr
&\hspace{20pt} -
e^{-2i(k+1)\vartheta}
\frac
{\qrfac{q e^{-2i\vartheta},  -\frac{2bc}{\lambda}e^{i\vartheta}, \frac{A}{2}e^{-i\vartheta}, \frac{2q}{A}e^{i\vartheta},
 \frac{B}{2}e^{i\vartheta}, \frac{2q}{B}e^{-i\vartheta}}{\infty}}
{\qrfac{q e^{2i\vartheta},  -\frac{2bc}{\lambda}e^{-i\vartheta}, \frac{A}{2}e^{i\vartheta}, \frac{2q}{A}e^{-i\vartheta},
 \frac{B}{2}e^{-i\vartheta}, \frac{2q}{B}e^{i\vartheta}}{\infty}}
\cr
&\hspace{30pt}
\cdot
 \ \!
 _{2}\phi_1
\left[ \begin{matrix} 
 -\frac{2bc}{\lambda}e^{-i\vartheta},  0
\\
q e^{-2i\vartheta}
\end{matrix}
;q, -\frac{\lambda q^{k+1}}{b}
\right]
\Bigg).
\end{align}
Next, we wish to examine whether the term in the brackets can be simplified by using a transformation formula. Indeed, on scanning the list of transformations in Gasper and Rahman, one finds in
\cite[Eq.~(III.31)]{GR90} a promising candidate. We take $a\mapsto -2bce^{i\vartheta}/\lambda$, 
$b\to 0$, $c\mapsto qe^{2i\vartheta}$ and $z\mapsto -\lambda q^{k+1}/b$ in this transformation formula to obtain:
\begin{align}\label{III.31-special}
\ \!
 _{2}\phi_1 &
\left[ \begin{matrix} 
 -\frac{2bc}{\lambda}e^{i\vartheta},  0
\\
q e^{2i\vartheta}
\end{matrix}
;q, -\frac{\lambda q^{k+1}}{b}
\right]
 \cr
& 
-
e^{-2i(k+1)\vartheta}
\frac
{\qrfac{q e^{-2i\vartheta},  -\frac{\lambda q}{2bc}e^{i\vartheta}, 2ce^{i\vartheta}, \frac{q}{2c}e^{-i\vartheta}
}{\infty}}
{\qrfac{q e^{2i\vartheta},  -\frac{\lambda q}{2bc}e^{-i\vartheta}, 2ce^{-i\vartheta}, \frac{q}{2c}e^{i\vartheta}
}{\infty}}
 \ \!
 _{2}\phi_1
\left[ \begin{matrix} 
 -\frac{2bc}{\lambda}e^{-i\vartheta},  0
\\
q e^{-2i\vartheta}
\end{matrix}
;q, -\frac{\lambda q^{k+1}}{b}
\right]
\cr
&\hspace{30pt} =
\frac
{\qrfac{e^{-2i\vartheta} }{\infty}}
{\qrfac{ -\frac{\lambda q}{2bc}e^{-i\vartheta}, 2cq^k e^{-i\vartheta}}{\infty}}
 \ \!
 _{1}\phi_1
\left[ \begin{matrix} 
 -{\lambda q}e^{i\vartheta}/{2bc}
\\
q^{1-k}e^{i\vartheta}
/{2c} 
\end{matrix}
;q, \frac{q^{1-k}e^{-i\vartheta}}{2c}
\right].
\end{align}
Now, comparing \eqref{III.31-special} and the two terms inside the bracket in \eqref{pk-sum}, we see that we should choose $B=4c$ and thus, since $A/B=-b/\lambda$, we must choose $A=-4bc/\lambda$.  

Next, we obtain \eqref{final-pka}. First we assume that $\Im(x)\ge 0$ or $\Im (\vartheta) \le 0$, so that 
$|e^{-i\vartheta}|\le 1$. 
Applying \eqref{III.31-special}, we find that \eqref{pk-sum} reduces to
\begin{align*}\label{pk-sum2}
g_k(x) &=
(1-q)
\frac{e^{i(k+1)\vartheta}}{2^{k+1}} 
\frac
{\qrfac{q, q e^{2i\vartheta}, e^{-2i\vartheta}}{\infty}}
{\qrfac{2ce^{i\vartheta}, \frac{q}{2c}e^{-i\vartheta}, 2cq^ke^{-i\vartheta}}{\infty}}\cr
&\hspace{20pt}
\times 
 \ \!
 _{1}\phi_1
\left[ \begin{matrix} 
 -{\lambda q}e^{i\vartheta}/{2bc}
\\
q^{1-k} e^{i\vartheta}/2c
\end{matrix}
;q, \frac{q^{1-k}e^{-i\vartheta}}{2c}
\right]
 .
\end{align*}
We can rewrite the $_1\phi_1$ on the right hand side using a special case of the transformation formula as a $_2\phi_1$ sum. The transformation we use is \cite[Eq.\ (III.4)]{GR90}:
\begin{equation}\label{III.4}
 \ \!
 _{2}\phi_1
\left[ \begin{matrix} 
 a,  b
\\
c
\end{matrix}
;q,z
\right]
=
\frac
{\qrfac{az}{\infty}}
{\qrfac{z}{\infty}}
\ \!
 _{2}\phi_2
\left[ \begin{matrix} 
 a,  c/b
\\
c, az 
\end{matrix}
;q, bz
\right]
.
\end{equation}
We use the $a\mapsto 0$, $b\mapsto -bq^{-k}/\lambda$, $c\mapsto q^{1-k} e^{i\vartheta}/2c$,
$z\mapsto -\lambda qe^{-i\vartheta}/2bc$ case of \eqref{III.4} and some elementary computations to write our solution of 
\eqref{mike-p-3term2} as follows. 
\begin{align*}
g_k(x) &=
(1-q) \frac{e^{ik\vartheta}\qrfac{2ce^{-i\vartheta}}{k}}{2^{k+2} (-i\sin\vartheta)}
\frac
{\qrfac{q, e^{2i\vartheta},  e^{-2i\vartheta},  -\frac{\lambda q}{2bc }e^{-i\vartheta}}{\infty}}
{\qrfac{ 2c e^{i\vartheta}, 2c e^{-i\vartheta},  \frac{q}{2c}e^{-i\vartheta}}{\infty}}\cr
&\hspace{20pt}
\times 
 \ \!
 _{2}\phi_1
\left[ \begin{matrix} 
 -{b}q^{-k}/{\lambda},  0
\\
q^{1-k} e^{i\vartheta}/2c
\end{matrix}
;q, -\frac{\lambda q}{2bc}e^{-i\vartheta}
\right].
\end{align*}
Finally, we divide through by some of the factors that do not depend on $k$, and obtain the solution 
$p_k(x)$ given in
 \eqref{final-pka}. Dividing \eqref{gk-qint} by these same factors, and inserting the values of $A$ and $B$, we obtain the  $q$-integral representation \eqref{final-pk-qint}. 
 
 To obtain \eqref{final-pkb}, we consider the case $\Im (x) \le 0$, replace $\vartheta$ by $-\vartheta$, and apply the transformations as above. Alternatively, we use a Heine 
 transformation \cite[Eq.\ (III.2)]{GR90}. 
 This completes the proof. 
\end{proof}

\section{The special case when $b=0$}\label{sec:specialcases}

In this section we consider the special case $b=0$ of \eqref{mikecf-3term}. Observe that other special cases which lead to Ramanujan's continued fractions (when $a=0$ or $\lambda =0$) can be treated as special cases of our work earlier in this paper. But when $b=0$, 
Blumenthal's theorem tells us that the measure has no absolutely continuous component, and is purely discrete. Thus this case has to be considered separately. 

When $b=0$, the continued fraction is 
 \begin{equation}
R(x)= 
\frac{1}{x+a}\fplus\frac{\lambda q}{x+aq}\fplus 
\frac{\lambda q^2}{x+aq^2}\fplus
\fdots  . \label{ram-jfrac}
\end{equation}
When $x=1$, it reduces to Ramanujan's continued fraction, given by the $b=0$ case of \eqref{mikecf}.
The corresponding three-term recurrence relation is 
\begin{equation}\label{ram-3term}
y_{k+1}(x) = (x+ aq^k) y_k(x) + \lambda q^{k} y_{k-1}(x), \text{ for } k > 0.
\end{equation}
By Proposition~\ref{cf-conv}, the numerator and denominator polynomials (denoted by $Q^*_k(x)$ and $Q_k(x)$, respectively) satisfy \eqref{ram-3term} and the initial values 
$$Q_0(x)=1, Q_1(x)=x+a; \; Q^*_0(x)=0, Q^*_1(x)= 1.$$ 
We require $0<|q|<1$ (with $q$ real),  $a\in\mathbb{R}$, and  $\lambda < 0$ to apply
Proposition~\ref{spectral}. 

Previously, Al--Salam and Ismail \cite{Al-I1983} had considered a very similar recurrence relation
\begin{equation*}\label{ai-3term}
U_{k+1} = x(1+ aq^k) U_k - \lambda q^{k-1} U_{k-1}, \text{ for } k > 0,
\end{equation*}
with $U_0=1$, $U_1=x(1+a)$. 

We denote the generating function of $Q_n(x)$  by $Q(t)$ and of $Q^*_n(x)$ by $Q^*(t)$. 
The generating functions are as follows. 
\begin{gather*} 
Q(t) = \sum_{k=0}^{\infty} \frac{\qrfac{-\lambda tq /a}{k} }{\qrfac{x t }{k+1}}
(at)^{k}q^{k\choose 2}, 
\cr
\intertext{and}
Q^*(t) = t\sum_{k=0}^{\infty} \frac{\qrfac{-\lambda qt /a}{k} }{\qrfac{xt }{k+1}}
(at)^{k}q^{\binom{k}{2}+k}. 
\end{gather*}

To obtain explicit expressions of the numerator and denominator polynomials, we need to extract the coefficient of powers of $t$.  
We need the $q$-binomial theorem in the form \cite[Ex. 1.2(vi)]{GR90}
\begin{equation}\label{gr90-ex-1.2vi}
\qrfac{at}{k}=\sum_{j\geq 0} \qbin{k}{j}_q (-1)^j q^{j\choose 2} (at)^j.
\end{equation}
In addition, we require the following special case of the $q$-binomial theorem (cf.~\cite[Eq.\ 1.3.2]{GR90})
valid for $|at|<1$:
\begin{equation}\label{q-bin-special-1} 
\frac{1}{\qrfac{at}{k+1}} =\sum_{m=0}^{\infty} \qbin{m+k}{k}
(at)^m.
\end{equation}

Using these, we find that $Q(t)$ can be written as
$$Q(t)=\sum_{j, k, m\geq 0} \qbin{k}{j}_{q} \qbin{k+m}{k}_{q}
a^{k-j} x^m \lambda^j  q^{{k\choose 2}+{j\choose 2}+j}t^{j+k+m}.$$
From here, we take the coefficient of $t^n$ to obtain an expression for $Q_n(x)$. We see that
\begin{align}\label{ram-den}
Q_n(x) &= 
\sum_{j, k \geq 0} \qbin{k}{j}_{q} \qbin{n-j}{k}_{q}
a^{k-j} x^{n-j-k} \lambda^j  q^{{k\choose 2}+{j\choose 2}+j} \cr
&= 
\sum_{j\geq 0} 
\qbin{n-j}{j}_q
\frac{\qrfac{-a/x}{n-j}}
{\qrfac{-a/x}{j}}
\lambda^j  x^{n-2j}  q^{j^2},
\end{align}
where we obtain the last equality by summing the inner sum using \eqref{gr90-ex-1.2vi}. Note that the first of these sums expresses $Q_n(x)$ as a polynomial in $x$ of degree $n$, since the indices satisfy 
$k+j\leq n$. 

Similarly, $Q_n^*(x)$ can be written as
\begin{align}\label{ram-n}
Q_n^*(x) &= 
\sum_{j\geq 0} 
\qbin{n-j-1}{j}_q
\frac{\qrfac{-a/x}{n-j}}
{\qrfac{-a/x}{j+1}}
\lambda^j  x^{n-2j-1}  q^{j^2+j}.
\end{align}

From Proposition \ref{spectral} and the comments on Blumenthal's theorem, we have the following orthogonality relation. 
\begin{Theorem}
Suppose  $q$ is real with $0<q<1$, $a\in \mathbb{R}$, and $\lambda<0$. Let $Q_n(x)$ be given 
by \eqref{ram-den}.
 Then we have the 
orthogonality relation
 $$\int_{-\infty}^{\infty} Q_n(x) Q_m(x) d\mu = (-\lambda)^n q^{\binom{n+1}{2}}\delta_{mn},$$
where $\mu$ is a purely discrete positive measure. 
\end{Theorem}
%

Next we find asymptotic formulas for the denominator and numerator polynomials, from the formulas for $Q_n(x)$ and $Q_n^*(x)$  above. 
We find that, for a fixed $x$, as $n\to\infty$,
\begin{gather*}
Q_n(x) \sim  
x^n \qrfac{-a/x}{\infty}
\ \!
 _{0}\phi_1
\left[ \begin{matrix} 
-
\\
-a/x
\end{matrix}
;q, \frac{\lambda q}{x^2}
\right]\cr
\intertext{and}
Q_n^*(x) \sim  
x^{n-1} \qrfac{-aq/x}{\infty}
\ \!
 _{0}\phi_1
\left[ \begin{matrix} 
-
\\
-aq/x
\end{matrix}
;q, \frac{\lambda q^2}{x^2}
\right].
\end{gather*}

Thus, the Stieltjes transform of $\mu$ is given by
$$
\int_{-\infty}^{\infty} \frac{d\mu(t)}{x-t} = \frac{1}{(x+a)} 
\frac{  
_{0}\phi_1
\left[ \begin{matrix} 
-
\\
-aq/x
\end{matrix}
;q, \displaystyle \frac{\lambda q^2}{x^2}
\right]
}{
 _{0}\phi_1
\left[ \begin{matrix} 
-
\\
-a/x
\end{matrix}
;q, \displaystyle \frac{\lambda q}{x^2}
\right]
},
$$
for $x\not\in \operatorname{supp}{\mu}$. 

\section{Formulas for the convergents}\label{sec:convergents}

In this section, we show how to obtain formulas for the convergents analogous to Ramanujan's Entry 16, which was highlighted in the introduction. We derive a formula given by Hirschhorn \cite{MDH1974}, and then take special cases corresponding to two of Ramanujan's continued fractions. We have recast Hirschhorn's original approach in terms of Proposition~\ref{cf-conv} in order to make it transparent how such formulas can be found. For some further examples, see Bowman, Mc Laughlin and Wyshinski~ \cite{BMW2006}.

We will require the notation of the  {\em $q$-multinomial coefficients}, defined as
$$\qbin{n}{k_1,k_2,\dots, k_r}_q = \frac{\qrfac{q}{n}}
{\qrfac{q}{k_1}\qrfac{q}{k_2}\cdots \qrfac{q}{k_r}\qrfac{q}{n-(k_1+k_2+\cdots+k_r)}}$$
where $n, k_1, k_2, \dots, k_r$ are positive integers and $n\geq k_1+k_2+\cdots + k_r$. 
When $n< k_1+k_2+\cdots +k_r$, we take the $q$-multinomial coefficient to be $0.$ When $r=1$, then these reduce to the $q$-binomial coefficients. 

We first consider \eqref{mikecf}.
%
Denote by $Y(t)$, $D(t)$ and $N(t)$ the generating functions of $y_k(x)$, $D_k(x)$ and $N_k(x)$ respectively. 
Multiply \eqref{mikecf-3term} by $t^{k+1}$ and sum over $k\geq 0$ to find that
$$(1-x(1-b)t-bt^2)Y(t) = y_0+ty_1 -xt(1-b)y_0-aty_0 +at(1+\lambda qt/a)Y(tq),$$
where we have used $y_0=y_0(x)$ and $y_1=y_1(x)$ to denote the initial values of $y_k(x)$.
Thus, the generating function of $D_n(x)$ satisfies the $q$-difference equation
$$D(t) = \frac{1}{1-x(1-b)t-bt^2} + \frac{at(1+\lambda tq /a)}{1-x(1-b)t-bt^2} D(tq).$$
Let $\alpha$ and $\beta$ be such that 
\begin{equation}\label{alpha-beta-mikecf}
1-(1-b)xt-bt^2 = (1-\alpha t) (1-\beta t).
\end{equation}
Using $\alpha$ and $\beta$ we can write the $q$-difference equation for $D(t)$ in a form that it can be iterated easily. As before, we obtain the generating function
\begin{equation*}
D(t) = \sum_{k=0}^{\infty} \frac{\qrfac{-\lambda tq /a}{k} }{\qrfac{\alpha t, \beta t }{k+1}}
(at)^{k}q^{k\choose 2}.\label{mikecf-denom-gf}
\end{equation*}
Similarly, we obtain the generating function of the numerators
\begin{equation*}
N(t)=t(1-b)\sum_{k=0}^{\infty} \frac{\qrfac{-\lambda tq /a}{k} }{\qrfac{\alpha t, \beta t }{k+1}}
(at)^{k}q^{{k\choose 2}+k}.
\end{equation*}
Notice that the $x$ is hidden implicitly in $\alpha$ and $\beta$. 

To obtain explicit formulas for the convergents, we need to find expressions for $N_n(x)$ and $D_n(x)$ when $x=1$.
Note that when $x=1$ in
\eqref{alpha-beta-mikecf}, then $\alpha =1$ and $\beta = -b$.

We use \eqref{gr90-ex-1.2vi} and \eqref{q-bin-special-1} 
to find that $D(t)$ with $\alpha =1$, $\beta=-b$ becomes
$$D(t)=\sum_{j, k,l, m\geq 0} \qbin{k}{j}_{q} \qbin{k+l}{k}_{q} \qbin{k+m}{k}_{q}
a^{k-j}(-b)^l \lambda^jq^{{k\choose 2}+{j\choose 2}+j}t^{j+k+l+m}.$$
We now take the coefficient of $t^n$ (so restrict the sum to $n=j+k+l+m$) to find that
\begin{align*}
D_n(1)&=\sum_{j, k,l \geq 0} \qbin{k}{j}_{q} \qbin{k+l}{k}_{q} \qbin{n-j-l}{k}_{q}
a^{k-j}(-b)^l \lambda^jq^{{k\choose 2}+{j\choose 2}+j} \cr
&= \sum_{j, k,l \geq 0} \qbin{k+l}{j,l}_{q} \qbin{n-j-l}{k}_{q}
a^{k-j}(-b)^l \lambda^jq^{{k\choose 2}+{j\choose 2}+j} .
\end{align*}
Similarly, we find that
\begin{equation*}
N_n(1) =(1-b) \sum_{j, k,l \geq 0} \qbin{k+l}{j,l}_{q} \qbin{n-j-l-1}{k}_{q}
a^{k-j}(-b)^l \lambda^jq^{{k\choose 2}+k+{j\choose 2}+j}.
\end{equation*}
We divide $N_{n+1}(1)$ by $(1-b)D_{n+1}(1)$ to obtain Hirschhorn's formula \cite{MDH1974}:
\begin{equation}
\frac{N_{n+1}(1)}{(1-b)D_{n+1}(1)} =
\frac{1}{1-b+a}\fplus\frac{b+\lambda q}{1-b+aq}\fplus 
\fdots \fplus \frac{b+\lambda q^n}{1-b+aq^n}. \label{mikecf-jfrac-conv}
\end{equation}
Taking $n\to\infty$ and invoking the two summations \eqref{gr90-ex-1.2vi} and \eqref{q-bin-special-1} we obtain Hirschhorn's formula for his infinite continued fraction as a ratio of two sums, under the condition $|b|<1$.

From \eqref{mikecf-jfrac-conv} we can take special cases $b=0$, $a=0$ or both to obtain results related to Ramanujan's continued fractions.
The first special case we consider is from the lost notebook \cite[Entry 6.3.1(iii)]{AB2005}
\begin{equation*}
\frac{1}{1-b}\fplus\frac{b+\lambda q}{1-b}\fplus\frac{b+\lambda q^2}{1-b}\fplus\frac{b+\lambda q^3}{1-b}\fplus\fdots.
\end{equation*}
This is obtained by taking $a=0$ in \eqref{mikecf}.
 Here is our formula for the convergents of \eqref{g-cfrac3}. We have,
\begin{equation}\label{g-cfrac3-n}
\frac{N^\prime_n}{D^\prime_n} =
\frac{1}{1-b}\fplus \frac{b+\lambda q}{1-b}\fplus
\frac{ b+\lambda q^2}{1-b}\fplus\frac{b+\lambda q^3}{1-b}
\fplus\fdots \fplus \frac{b+\lambda q^n}{1-b}
 ,
\end{equation}
where the numerator and denominator polynomials of the $(n+1)$th convergent are given by:
\begin{equation*}
N^\prime_n = 
\sum_{k, j\geq 0} q^{k^2+k}\lambda^k
\qbin{k+j}{k}_{q} \qbin{n-k-j}{k}_{q} (-b)^j
\end{equation*}
and
\begin{equation*}
D^\prime_n = 
\sum_{k, j\geq 0} q^{k^2}\lambda^k 
\qbin{k+j}{k}_{q} \qbin{n-k-j+1}{k}_{q} (-b)^j
.
\end{equation*}

When $b=0$, this immediately reduces to  \eqref{entry16}, Ramanujan's Entry 16.
Upon taking $n\to\infty$, we obtain Ramanujan's continued fraction evaluation, given in
Andrews and Berndt \cite[Entry 6.2.1(iii)]{AB2005}. 
We define
\begin{align*}
g(b,\lambda)&:=\sum_{k=0}^{\infty}\frac{ \lambda^k q^{k^2}}{\qrfac{q}{k}\qrfac{-bq}{k}}. 
\end{align*}
Then
 for $|b|<1$,
\begin{equation}
\frac{g(b,\lambda q)}{g(b,\lambda)}
=\frac{1}{1-b}\fplus\frac{b+\lambda q}{1-b}\fplus\frac{b+\lambda q^2}{1-b}\fplus\frac{b+\lambda q^3}{1-b}\fplus\fdots.\label{g-cfrac3}
\end{equation}

The condition $|b|<1$ appears quite naturally as a requirement for the sum to be convergent. 
To see this, consider the limit
\begin{align*}
\lim_{n\to\infty} N^\prime_n &= \lim_{n\to\infty} 
\sum_{k, j\geq 0} q^{k^2+k}\lambda^k \frac{\qrfac{q}{k+j} \qrfac{q}{n-k-j}}
{\qrfac{q}{k}\qrfac{q}{j}  \qrfac{q}{n-2k-j} \qrfac{q}{k}}
 (-b)^j\cr
 &= \sum_{k\geq 0}  \frac{q^{k^2+k}\lambda^k}
{ \qrfac{q}{k}}
\sum_{j\geq 0} 
\frac{\qrfac{q}{k+j} }
{\qrfac{q}{j}\qrfac{q}{k} }
 (-b)^j\cr
 &= \sum_{k\geq 0}  \frac{q^{k^2+k}\lambda^k}
{ \qrfac{q}{k}\qrfac{-b}{k+1}} ,
\end{align*}
upon invoking \eqref{q-bin-special-1}, assuming $|b|<1$. This shows that
$$ \lim_{n\to\infty} N^\prime_n =
\frac{g(b,\lambda q)}{1+b}.$$
Similarly, we can see that
$$ \lim_{n\to\infty} D^\prime_n =
\frac{g(b,\lambda )}{1+b}, $$
and this completes a proof of \eqref{g-cfrac3}.

Next we take $b=0$ in \eqref{mikecf}.
Ramanujan found the continued fraction (see Entry 15 of \cite[ch.\ 16 ]{Berndt1991-RN3} or \cite[Entry 6.3.1(ii)]{AB2005})
\begin{align*}
\frac{g(a,\lambda )}{g(a,\lambda q)}
&=
1+\frac{\lambda q}{1+aq}\fplus\frac{\lambda q^2}{1+aq^2}\fplus\frac{\lambda q^3}{1+aq^3}\fplus\fdots .
\end{align*} 
A formula for the convergents of Ramanujan's Entry 15 is as follows. Let
\begin{gather*}
\widehat{N}_n = 
\sum_{j\geq 0} q^{j^2}\lambda^j 
\qbin{n+1-j}{j}_{q} 
\frac{\qrfac{-aq}{n-j}}{\qrfac{-a}{j}} \cr
\intertext{and}
\widehat{D}_n = 
\sum_{j\geq 0} q^{j^2+j}\lambda^j
\qbin{n-j}{j}_{q}
\frac{\qrfac{-aq}{n-j}}{\qrfac{-aq}{j}}.
\end{gather*}
Then, for $n=1, 2, 3, \dots$, we have
\begin{equation}
(1+a)\frac{\widehat{N}_n}{\widehat{D}_n} =
1+a + \frac{\lambda q}{1+aq}\fplus
\frac{\lambda q^2}{1+aq^2}\fplus\frac{\lambda q^3}{1+aq^3}
\fplus\fdots \fplus \frac{\lambda q^n}{1+aq^n}
 .\label{entry16-gen1-a}
\end{equation}
%
To obtain \eqref{entry16-gen1-a}, we take $x=1$ in \eqref{ram-den} and \eqref{ram-n} and observe that
\begin{gather*}
Q_{n+1}(1)=(1+a)\widehat{N}_n\cr
\intertext{and}
Q_{n+1}^*(1)=\widehat{D}_n.
\end{gather*}

When $a=0$, \eqref{entry16-gen1-a} reduces to Ramanujan's Entry 16 given in \eqref{entry16}. Formula 
\eqref{entry16-gen1-a} is implicit in Al-Salam and Ismail's study \cite{Al-I1983} of the orthogonal polynomials associated with Rogers--Ramanujan continued fraction.  Bhatnagar and Hirschhorn \cite{BH2016} wrote it in this form and gave an elementary proof following Euler's approach given in 
\cite{GB2014}. 

Formulas  \eqref{entry16-gen1-a} and \eqref{g-cfrac3-n} are generalizations of Ramanujan's Entry 16, corresponding to two extensions of the Rogers--Ramanujan continued fraction given by Ramanujan in the Lost Notebook, recorded as Entry 6.3.1(ii) and (iii), respectively in 
\cite{AB2005}. As we have seen, such formulas can be discovered quite easily using generating functions.

\subsection*{Acknowledgments} This work was done at the sidelines of many workshops, conferences and summer schools organized by the members of the Orthogonal Polynomials and Special Functions (OPSF) group of SIAM. We thank the organizers of the following:  OPSF summer school, (July 2016), University of  Maryland; the international conference on special functions: theory, computation and applications, (June 2018), Liu Bie Ju center for mathematical sciences, City University of Hong Kong, Hong Kong; and,
summer research institute on $q$-series, (July-Aug 2018), Chern Institute of Mathematics, Nankai University, Tianjin, PR China. Finally, we wish to thank the anonymous referees for several useful suggestions and corrections.

Research of the first named author was supported by grants of the Austrian Science Fund (FWF): START grant Y463 and FWF  grant F50-N15.

\subsection*{Data availability}
Data sharing not applicable to this article as no datasets were generated or analyzed during the current study.

%

\begin{thebibliography}{10}

\bibitem{Al-I1983}
W.~A. Al-Salam and M.~E.~H. Ismail.
\newblock Orthogonal polynomials associated with the {R}ogers-{R}amanujan
  continued fraction.
\newblock {\em Pacific J. Math.}, 104(2):269--283, 1983.

\bibitem{AAR1999}
G.~E. Andrews, R.~Askey, and R.~Roy.
\newblock {\em Special functions}, volume~71 of {\em Encyclopedia of
  Mathematics and its Applications}.
\newblock Cambridge University Press, Cambridge, 1999.

\bibitem{AB2005}
G.~E. Andrews and B.~C. Berndt.
\newblock {\em Ramanujan's {L}ost {N}otebook. {P}art {I}}.
\newblock Springer, New York, 2005.

\bibitem{AI1984}
R.~Askey and M.~Ismail.
\newblock Recurrence relations, continued fractions, and orthogonal
  polynomials.
\newblock {\em Mem. Amer. Math. Soc.}, 49(300):iv+108, 1984.

\bibitem{Berndt1991-RN3}
B.~C. Berndt.
\newblock {\em Ramanujan's {N}otebooks. {P}art {III}}.
\newblock Springer-Verlag, New York, 1991.

\bibitem{BA1984}
S.~Bhargava and C.~Adiga.
\newblock On some continued fraction identities of {S}rinivasa {R}amanujan.
\newblock {\em Proc. Amer. Math. Soc.}, 92(1):13--18, 1984.

\bibitem{GB2014}
G.~Bhatnagar.
\newblock How to prove {R}amanujan's {$q$}-continued fractions.
\newblock In {\em Ramanujan 125}, volume 627 of {\em Contemp. Math.}, pages
  49--68. Amer. Math. Soc., Providence, RI, 2014.

\bibitem{BH2016}
G.~{Bhatnagar} and M.~D. {Hirschhorn}.
\newblock {A formula for the convergents of a continued fraction of Ramanujan}.
\newblock {\em ArXiv e-prints}, Mar. 2016.
\newblock \href{https://arxiv.org/abs/1603.07664}{arXiv:1606.07664}.

\bibitem{BMW2006}
D.~Bowman, J.~Mc~Laughlin, and N.~J. Wyshinski.
\newblock A {$q$}-continued fraction.
\newblock {\em Int. J. Number Theory}, 2(4):523--547, 2006.

\bibitem{Chihara1978}
T.~S. Chihara.
\newblock {\em An introduction to orthogonal polynomials}.
\newblock Gordon and Breach Science Publishers, New York-London-Paris, 1978.
\newblock Mathematics and its Applications, Vol. 13.

\bibitem{GR90}
G.~Gasper and M.~Rahman.
\newblock {\em Basic {H}ypergeometric {S}eries}, volume~96 of {\em Encyclopedia
  of Mathematics and its Applications}.
\newblock Cambridge University Press, Cambridge, second edition, 2004.
\newblock With a foreword by Richard Askey.

\bibitem{MDH1974}
M.~D. Hirschhorn.
\newblock A continued fraction.
\newblock {\em Duke Math. J.}, 41:27--33, 1974.

\bibitem{MI2009}
M.~E.~H. Ismail.
\newblock {\em Classical and quantum orthogonal polynomials in one variable},
  volume~98 of {\em Encyclopedia of Mathematics and its Applications}.
\newblock Cambridge University Press, Cambridge, 2009.
\newblock With two chapters by Walter Van Assche, With a foreword by Richard A.
  Askey, Reprint of the 2005 original.

\bibitem{IS1997}
M.~E.~H. Ismail and D.~Stanton.
\newblock Classical orthogonal polynomials as moments.
\newblock {\em Canad. J. Math.}, 49(3):520--542, 1997.

\bibitem{IS1998}
M.~E.~H. Ismail and D.~Stanton.
\newblock More orthogonal polynomials as moments.
\newblock In {\em Mathematical essays in honor of {G}ian-{C}arlo {R}ota
  ({C}ambridge, {MA}, 1996)}, volume 161 of {\em Progr. Math.}, pages 377--396.
  Birkh\"auser Boston, Boston, MA, 1998.

\bibitem{IS2002}
M.~E.~H. Ismail and D.~Stanton.
\newblock {$q$}-integral and moment representations for {$q$}-orthogonal
  polynomials.
\newblock {\em Canad. J. Math.}, 54(4):709--735, 2002.

\bibitem{IS2006}
M.~E.~H. Ismail and D.~Stanton.
\newblock Ramanujan continued fractions via orthogonal polynomials.
\newblock {\em Adv. Math.}, 203(1):170--193, 2006.

\bibitem{PN1979}
P.~G. Nevai.
\newblock Orthogonal polynomials.
\newblock {\em Mem. Amer. Math. Soc.}, 18(213):v+185, 1979.

\bibitem{TJS1894}
T.~J. Stieltjes.
\newblock Recherches sur les fractions continues.
\newblock {\em Ann. Fac. Sci. Toulouse Math. (6)}, 4(3):J76--J122, 1995.
\newblock Reprint of Ann. Fac. Sci. Toulouse {{\bf{8}}} (1894), J76--J122.

\bibitem{TJS1895}
T.~J. Stieltjes.
\newblock Recherches sur les fractions continues.
\newblock {\em Ann. Fac. Sci. Toulouse Math. (6)}, 4(4):A5--A47, 1995.
\newblock Reprint of Ann. Fac. Sci. Toulouse {{\bf{9}}} (1895), A5--A47.

\end{thebibliography}
%

\end{document}